\documentclass[11pt]{amsart}
\usepackage{amsmath,amssymb,amsthm,upref,graphicx,mathrsfs, enumerate,citeref}
\usepackage{esint}
\usepackage{textcomp}
\usepackage{array}
\usepackage{color,xcolor}
\usepackage{latexsym}
\usepackage{amsmath}
\usepackage{amssymb}
\usepackage{amsthm} 
\usepackage{amscd}
\usepackage{color}
\usepackage{comment}

\numberwithin{equation}{section}
\allowdisplaybreaks

\newtheorem{theorem}{Theorem}[section]

\newtheorem{lemma}[theorem]{Lemma}

\theoremstyle{definition}
\newtheorem{definition}[theorem]{Definition}

\newtheorem{remark}[theorem]{Remark}

\begin{document}

\title[Density theorems in neural network with variable exponent]
{Some density theorems in neural network \\
with variable exponent}

\author{Mitsuo Izuki, Takahiro Noi, Yoshihiro Sawano 
and Hirokazu Tanaka}
\address{
Faculty of Liberal Arts and Sciences\\
Tokyo City University\\
1-28-1, Tamadutsumi Setagaya-ku, Tokyo 158-8557, Japan. 
}
\email{izuki@tcu.ac.jp}

\address{Department of Mathematical and Data Science\\
Otemon Gakuin University \\
2-1-15 Nishiai ibaraki, Osaka 567-8502, Japan.
}
\email{taka.noi.hiro@gmail.com}

\address{Department of Mathematics\\ Graduate School of Science and Engineering,\ Chuo University\\ 
13-27 Kasuga 1-chome, Bunkyoku, Tokyo 112-8551, Japan}
\email{yoshihiro-sawano@celery.ocn.ne.jp}

\address{Faculty of Information Technology
\\
Tokyo City University\\
1-28-1, Tamadutsumi Setagaya-ku, Tokyo 158-8557, Japan. 
}
\email{htanaka@tcu.ac.jp}

\maketitle

\begin{abstract}
In this paper, we extend several approximation theorems, originally formulated in the context of the standard $L^p$ norm, to the more general framework of variable exponent spaces. Our study is motivated by applications in neural networks, where function approximation plays a crucial role. In addition to these generalizations, we provide alternative proofs for certain well-known results concerning the universal approximation property. In particular, we highlight spaces with variable exponents as illustrative examples, demonstrating the broader applicability of our approach.
\end{abstract}



\section{Introduction}

Neural networks serve as \emph{functional approximators}, particularly in the context of \emph{supervised learning}, where the goal is to learn from a training dataset consisting of input-output pairs. Given such a dataset, a neural network adjusts its internal weights to minimize the discrepancy between the network's output and the correct output.
After training, the performance of the neural network is evaluated using a separate \emph{test dataset} to assess its ability to \emph{generalize} to unseen data.
There are two commonly used cost functions, depending on the nature of the task:
\begin{itemize}
    \item \textbf{Cross-entropy loss} is typically employed for \emph{classification} tasks, where the output is a probability distribution over classes.
    \item \textbf{Mean squared error (MSE)}, corresponding to the squared $L^2$ norm, is standard in \emph{regression} problems, where the outputs are continuous values.
\end{itemize}
The MSE cost function assumes that the errors follow a \emph{Gaussian distribution}, a premise well-supported by the classical \emph{central limit theorem}. However, in many real-world applications, the error distribution may not conform to the Gaussian model and often includes \emph{outliers}.
MSE is particularly \emph{sensitive to outliers} due to its squaring of errors. This sensitivity can lead to \emph{overfitting}, where the neural network becomes overly tuned to the noise or outliers in the training data. As a result, its performance on the test dataset may deteriorate, indicating poor generalization.

Overfitting is a well-known problem in machine learning, and several solutions have been proposed to mitigate it, including the use of alternative cost functions based on $L^1$ and $L^\infty$ norms. The $L^1$ norm grows more slowly than the $L^2$ norm, making it less sensitive to outliers. In the field of compressed sensing, the $L^1$ norm is known to yield unique and sparse solutions to underdetermined linear systems~\cite{eldar2012compressed}.

The \( L^\infty \) norm, which represents the maximum error, provides a means to control the worst-case error in a training dataset. In \( H^\infty \) control theory--which originated from functional calculus at the intersection of complex and functional analysis--the \( L^\infty \) norm plays a critical role in designing robust controllers by minimizing the maximum effect of perturbations on system outputs~\cite{simon2006optimal}.

Previous studies have established the universal approximation theorem for multilayer neural networks under the $L^p$ norm, uniformly over the entire input space. However, real-world applications often demand more nuanced error control. In some regions of the input space, it may be desirable to suppress the influence of outliers, while in other regions, controlling the maximum possible error is more important.

This motivates the extension of the universal approximation theorem to variable exponent Lebesgue spaces, where the error norm can vary as a function of the input location. Such generalizations would allow for greater flexibility in the design and analysis of neural networks, enabling them to meet diverse performance criteria across different regions of the input domain.

This paper presents a deeper analysis of existing results related to the universal approximation property, aiming to extend and refine the theoretical framework surrounding this concept. In particular, we introduce new perspectives and techniques that enhance our understanding of the scope and limitations of universal approximation. As an application, we examine variable exponent Lebesgue spaces, which have garnered significant attention for their flexibility in modeling diverse real-world phenomena. By investigating the interplay between the universal approximation property and the structure of these spaces, we offer new insights into their functional and approximation-theoretic characteristics. This work builds upon and reinforces the results in \cite{arxiv2020}.

We note that Cybenko~\cite{Cybenko} introduced a functional-analytic approach to approximation theorems based on the Hahn--Banach theorem and the Riesz representation theorem. This approach was subsequently developed in the works of Hornik~\cite{Hornik1991} and Park and Sandberg~\cite{Park1993}.

In what follows, we employ the Hahn--Banach theorem in the following form~\cite[page 382, Theorem 15.72]{Yeh}:
\begin{theorem}
Let $(V, \|\cdot\|)$ be a complex normed space, and let $(V^*, \|\cdot\|_*)$ denote its dual space. Suppose $M$ is a proper linear subspace of $V$, 
and let $v_0 \in V$. Write
\[
d := \inf_{u \in M} \|v_0 - u\| > 0.
\]
Then there exists a bounded linear functional $f : V \to \mathbb{K}$ such that
\[
M \subset \ker(f), \quad f(v_0) = d, \quad \text{and} \quad \|f\|_* = 1.
\]
\end{theorem}

Here and below, 
we use the following notation:
\begin{itemize}
    \item Given a measurable set \( E \), the characteristic function of \( E \) is denoted by \( \chi_E \).
    \item The space \( C_{\rm c}({\mathbb R}^n) \) refers to the set of all continuous functions on \( {\mathbb R}^n \) with compact support.
    \item The space \( C_{\rm c}^{\infty}({\mathbb R}^n) \) denotes the set of all infinitely differentiable functions on \( {\mathbb R}^n \) with compact support.
\end{itemize}

Here we describe the remaining part of this paper.
We state preliminary facts in Section \ref{s2}.
Section \ref{s3} collects our main results.

\section{Preliminaries}
\label{s2}

In this section, we collect some preliminary facts.
Section \ref{s2.1} introduces Lebesgue spaces with variable exponents, and Section \ref{s2.2} reviews the duality result.

\subsection{Lebesgue spaces with variable exponent}
\label{s2.1}

Let $\mu$ be a measure on $\mathbb{R}^n$, and let $p: \mathbb{R}^n \to [1,\infty]$ be a $\mu$-measurable function. We define the following sets:

\begin{align}
    \Omega_{\infty} &= \{ x \in \mathbb{R}^n : p(x) = \infty \}, \label{eq:omega_infty} \\
    \Omega_{1} &= \{ x \in \mathbb{R}^n : p(x) = 1 \}, \label{eq:omega_1} \\
    \Omega_{*} &= \{ x \in \mathbb{R}^n : 1 < p(x) < \infty \}. \label{eq:omega_star}
\end{align}

We now recall the definition of the modular.

\begin{definition}
Given a measurable function $f$, the modular is defined as
\begin{equation}
    \rho_{p(\cdot)}(f) = \int_{\Omega_1 \cup \Omega_*} |f(x)|^{p(x)} \, \mathrm{d}\mu(x) + \| f \|_{L^{\infty}(\Omega_{\infty})}.
    \label{eq:modular}
\end{equation}
The variable exponent Lebesgue space $L^{p(\cdot)}(\mathrm{d}\mu)$ consists of all measurable functions $f$ such that
\begin{equation}
    \rho_{p(\cdot)}(f/\lambda) < \infty \quad \text{for some } \lambda > 0.
    \label{eq:membership_condition}
\end{equation}
The norm in $L^{p(\cdot)}(\mathrm{d}\mu)$ is given by
\begin{equation}
    \| f \|_{L^{p(\cdot)}(\mathrm{d}\mu)} = \inf \left\{ \lambda > 0 : \rho_{p(\cdot)}(f/\lambda) \leq 1 \right\}.
    \label{eq:norm}
\end{equation}
If $\mu$ is the Lebesgue measure, we write $L^{p(\cdot)} := L^{p(\cdot)}(\mathrm{d}\mu)$.
\end{definition}

If the variable exponent 
$p(\cdot)$ equals to a constant $p$,
then $L^{p(\cdot)}$ is the usual Lebesgue spae $L^p$.
We present the H\"older inequality in Lebesgue spaces with a variable exponent. Let $p : \mathbb{R}^n \to [1,\infty]$ be a measurable function. We
write its essential infimum and supremum of $p(x)$ as follows:
\begin{equation}\label{eq:p_min_max}
    p_- = \mathrm{ess\, inf}_{x \in \mathbb{R}^n} p(x), \quad
    p_+ = \mathrm{ess\, sup}_{x \in \mathbb{R}^n} p(x).
\end{equation}
Furthermore, we denote by $p'(\cdot)$ the conjugate exponent,
that is $p(\cdot)$ satisfies
\begin{equation}\label{eq:conjugate_exponent}
    \frac{1}{p(x)} + \frac{1}{p'(x)} = 1 \quad (x \in {\mathbb R}^n).
\end{equation}

The following two lemmas present results concerning fundamental density and approximation in variable exponent Lebesgue spaces $L^{p(\cdot)}({\rm d}\mu)$.

\begin{lemma}[Corollary 2.73 in \cite{CF-book}]
\label{lemma:variable-dense}
Let $p(\cdot) : \mathbb{R}^n \to [1,\infty]$ be a measurable function such that $p_+ < \infty$. Then the space of compactly supported continuous functions $C_{\rm c}(\mathbb{R}^n)$ is dense in $L^{p(\cdot)}({\rm d}\mu)$.
\end{lemma}

\begin{lemma}[Theorem 5.11 in \cite{CF-book}]
\label{lem:approx-convolution}
Let $\phi \in L^1(\mathbb{R}^n)$ be such that
\begin{equation}
    \int_{\mathbb{R}^n} \phi(x){\rm d}x = 1.
    \label{eq:phi-normalization}
\end{equation}
For $\sigma > 0$, define
\begin{equation}
    \phi_\sigma(x) = \sigma^{-n} \phi\left( \frac{x}{\sigma} \right), \quad x \in \mathbb{R}^n.
    \label{eq:phi-sigma-def}
\end{equation}
Suppose that $p(\cdot): \mathbb{R}^n \to [1,\infty)$ satisfies $p_+ < \infty$. Then, for all $f \in L^{p(\cdot)}(\mathbb{R}^n)$, we have
\begin{equation}
    \lim_{\sigma \to 0^+} \left\| \phi_\sigma * f - f \right\|_{L^{p(\cdot)}} = 0.
    \label{eq:approx-convergence}
\end{equation}
\end{lemma}

\begin{remark}
Due to its role in (\ref{eq:approx-convergence}), the system 
$\{\phi_\sigma\}_{\sigma > 0}$ is referred to as an approximation of the identity or an approximate identity. The concept of an identity has been studied in various function spaces (see \cite{Schwartz2021, MMO2010, MMOS2013, approxHerz2017}). As an application, this framework allows us to establish universal approximation theorems in these spaces.
\end{remark}

\subsection{Duality in Lebesgue spaces with variable exponent}
\label{s2.2}

We begin by recalling H\"{o}lder's inequality.

\begin{lemma}[Theorem 2.26 in \cite{CF-book}]
\label{lem:250326-2}
Let $p(\cdot) : \mathbb{R}^n \to [1,\infty]$ be a measurable function. For all $f\in L^{p(\cdot)}({\rm d}\mu)$ and $g\in L^{p'(\cdot)}({\rm d}\mu)$, we have $fg\in L^1({\rm d}\mu)$ and
\begin{equation}\label{eq:holder-inequality}
    \| fg \|_{L^1({\rm d}\mu)} \le K_{p(\cdot)} \| f \|_{L^{p(\cdot)}({\rm d}\mu)} \| g \|_{L^{p'(\cdot)}({\rm d}\mu)},
\end{equation}
where
\begin{equation}\label{eq:Kp-definition}
    K_{p(\cdot)} = \left(\frac{1}{p_-} -\frac{1}{p_+} +1\right)\|  \chi_{\Omega_{*}} \|_{L^{\infty}}
    +\|  \chi_{\Omega_{\infty}} \|_{L^{\infty}}
    +\|  \chi_{\Omega_{1}} \|_{L^{\infty}}.
\end{equation}
\end{lemma}

The dual space $L^{p(\cdot)}({\rm d}\mu)^*$ consists of all bounded and linear functionals on $L^{p(\cdot)}({\rm d}\mu)$. In particular, for a given function $g \in L^{p'(\cdot)}({\rm d}\mu)$, we define the functional
\begin{equation}\label{eq:dual-functional}
    F_g(f) = \int_{\mathbb{R}^n} f(x) g(x) \, {\rm d}\mu(x), \quad \text{for } f \in L^{p(\cdot)}({\rm d}\mu).
\end{equation}
By virtue of H\"{o}lder's inequality (Lemma \ref{lem:250326-2}), it follows that $F_g \in L^{p(\cdot)}({\rm d}\mu)^*$.

The following lemma presents further duality properties and the Riesz representation theorem for Lebesgue spaces with variable exponent, as established in \cite[Proposition 2.79 and Theorem 2.80]{CF-book}.

\begin{lemma}
    Let $p(\cdot) : \mathbb{R}^n \to [1,\infty]$ be a measurable function. Then the following statements hold:
    
    \begin{enumerate}
        \item[{\rm (I)}] Let $g$ be a measurable function. The following conditions are equivalent{\rm :}
        
        \begin{enumerate}
            \item[{\rm (i)}] $F_g \in L^{p(\cdot)}({\rm d}\mu)^*$.
            \item[{\rm (ii)}] $g \in L^{p'(\cdot)}({\rm d}\mu)$.
        \end{enumerate}
        
        Furthermore, in this case, there exists a constant $C \geq 1$, depending only on $p(\cdot)$, such that
        \begin{equation}\label{eq:dual_norm_equivalence}
            C^{-1} \|g\|_{L^{p'(\cdot)}({\rm d}\mu)}
            \leq\|F_g\|_{L^{p(\cdot)}({\rm d}\mu)^*}
            \leq C \|g\|_{L^{p'(\cdot)}({\rm d}\mu)}.
        \end{equation}
        
        \item[{\rm (II)}] The following conditions are equivalent{\rm :}
        
        \begin{enumerate}
            \item[{\rm (iii)}] $p_+ < \infty$.
            \item[{\rm (iv)}] For all $F \in L^{p(\cdot)}({\rm d}\mu)^*$, there exists a unique function $g \in L^{p'(\cdot)}({\rm d}\mu)$ such that $F = F_g$.
        \end{enumerate}
    \end{enumerate}
\end{lemma}

\section{Main results}
\label{s3}

We focus on two results concerning $L^p$-approximation and extend them to the setting of variable exponent spaces. One of our results also holds in Herz spaces.

Section \ref{s3.1} revisits a result by Hornik, Stinchcombe, and White.
Section \ref{s3.2} examines a result by Park and Sandberg in the framework of variable exponent Lebesgue spaces.
Section \ref{s3.3} deals with Herz spaces as an application of the results
in Section \ref{s3.2}.
Finally, Section \ref{s3.4} considers the modular inequality,
where we show that this inequality fails for variable exponents.
\subsection{Hornik--Stinchcombe--White 
\cite{HSW1989} (1989)}
\label{s3.1}
Following \cite{HSW1989}, we introduce some definitions, notation, and concepts.

\begin{definition}[Squashing Function]
A function $\Psi: \mathbb{R} \to [0,1]$ is called a squashing function if it satisfies the following properties:
\begin{enumerate}
    \item $\Psi$ is non-decreasing,
    \item $\displaystyle \lim_{x\to\infty} \Psi(x) = 1$,
    \item $\displaystyle \lim_{x\to-\infty} \Psi(x) = 0$.
\end{enumerate}
\end{definition}

Since squashing functions are non-decreasing, they are measurable.

The following lemma states a well-known property of squashing functions
and some generalized results are proved.
For example,
Cybenko
\cite[Lemma 1]{Cybenko}
has proved the lemma for bounded and measurable sigmoidal functions.
Hornik
\cite[Theorem 5]{Hornik1991}
has proved it for non-constant and bounded functions.


\begin{lemma}\label{lem:squashing}
Let $\Psi$ be a squashing function. If a Borel signed measure $\mu$ satisfies
\begin{equation}\label{eq:integral}
    \int_{{\mathbb R}^n} \Psi(w \cdot x + b) {\rm d}\mu(x) = 0
\end{equation}
for all $w \in \mathbb{R}^n$ and $b \in \mathbb{R}$, then $\mu = 0$.
\end{lemma}

We provide an original proof
using the following fact on the integrability
of the Fourier transform of bounded functions:
\begin{lemma}{\rm \cite[Theorem 2]{OS}}
\label{lem:250401-1}
Let $f: (0,\infty) \to [0,\infty)$ be a differentiable function such that $f(\infty) = f'(\infty) = 0$ and $f(0^+), f'(0^+) \in \mathbb{R}$, and assume that $f$ is of bounded variation. Then 
\[
\int_{{\mathbb R}}
\left|\int_{{\mathbb R}}
f(x)e^{i x \cdot \xi}{\rm d}x\right|{\rm d}\xi<\infty.
\] 
\end{lemma}


\begin{proof}[Proof of Lemma \ref{lem:squashing}]
Let $\eta \in C^\infty_{\rm c}({\mathbb R})$ be such that
$\Psi*\eta$ is not zero.
Then
\[
\lim_{t \to \infty}\Psi*\eta(t)=\int_{{\mathbb R}}\eta(t){\rm d}t
\]
and
\[
\lim_{t \to \infty}\Psi*\eta(-t)=0.
\]
Then 
\[
\int_{{\mathbb R}^n}\eta*\Psi(w \cdot x+b){\rm d}\mu(x)=0.
\]
Set
\[
\Phi=\eta*\Psi+\eta*\Psi(-\cdot)-\int_{{\mathbb R}}\eta(t){\rm d}t.
\]
Then
\[
\lim_{r \to \infty}\Phi^{(l)}(r)=0
\]
for all $l=0,1,\ldots$
and
\[
\int_{{\mathbb R}^n}\Phi(w \cdot x+b){\rm d}\mu(x)=0.
\]
In fact, if $l=0$, then this is trivial.
Otherwise use
\[
\lim_{t \to \infty}\eta^{(l)}*\Psi(t)=
\int_{{\mathbb R}}\eta^{(l)}(t){\rm d}t=0.
\]
Furthermore,
since
\[
\Psi(b)\mu({\mathbb R}^n)
=
\lim_{w \to 0}
\int_{{\mathbb R}^n}\Phi(w \cdot x+b){\rm d}\mu(x)=0,
\]
we have
\[
\int_{{\mathbb R}^n}\Phi(w \cdot x+b){\rm d}\mu(x)
=
\int_{{\mathbb R}^n}\eta*\Phi(w \cdot x+b){\rm d}\mu(x)
+
\int_{{\mathbb R}^n}\eta*\Phi(-w \cdot x-b){\rm d}\mu(x)
-k\mu({\mathbb R}^n)=0
\]
where
$\displaystyle k:=\int_{{\mathbb R}} \eta(t)\,{\rm d}t$.
Thus, 
$\Phi={\mathcal F}\Psi_0$
for some $\Psi_0 \in L^1({\mathbb R}) \setminus \{0\}$,
since
$\Phi$ is an even function
due to Lemma \ref{lem:250401-1}.
Then
\[
\int_{\mathbb{R}^n} \Phi(a\cdot x-\theta) \, {\rm d}\mu (x)
=
\int_{\mathbb{R}^n}\int_{{\mathbb R}} e^{-i t a\cdot x+i t\theta}\Psi_0(t){\rm d}t \, {\rm d}\mu (x).
\]
If we write
\[
\Phi_a(t)=\int_{{\mathbb R}^n} e^{-i t a\cdot x}{\rm d}\mu (x),
\]
then we have
\[
{\mathcal F}^{-1}[\Psi_0 \cdot \Phi_a](\theta)=0
\]
for all $\theta \in {\mathbb R}$.
Thus,
$\Psi_0(t)\Phi_a(t)=0$.
Since we can replace
$\Psi_0(t)$ with $\Psi(b t)$ for all $b>0$,
we have
$\Psi(b t)\Phi_a(t)=0$ for all $b>0$.
Thus, $\Phi_a(t)=0$ and hence $\mu=0$.
\end{proof}

Given a squashing function $\Psi$
we define the following class $\Sigma(\Psi)$.

\begin{definition}\label{defi:250326-1}
The set $\Sigma(\Psi)$ consists
of all functions 
$f: {\mathbb R}^n \to {\mathbb R} $
that can be expressed in the form 
\[
 f(x) = \sum_{j=1}^{N} w_j \Psi(a_j\cdot x+b_j) \quad (x \in {\mathbb R}^n), 
\]
where $N \in \mathbb{N}$,
$w_j \in {\mathbb R}$, 
$a_j\in{\mathbb R}^n$, and 
$b_j\in{\mathbb R}$ ($j=1,2, \cdots , N$).
Here, 
$a_j \cdot x$ denotes the Euclidean inner product of $a_j\in{\mathbb R}^n$ and $x\in{\mathbb R}^n$.
\end{definition}

In the context of probability measures, Corollary 2.2 in \cite{HSW1989} establishes foundational results under fixed exponents. This corollary can be extended to accommodate variable exponents.

\begin{theorem}  \label{prop 20250223-1-noi}
Let $\mu$ be a probability measure supported on a compact set $K \subset \mathbb{R}^n$. Let $p: \mathbb{R}^n \to [1, \infty)$ be a Borel-measurable function such that
\begin{equation}
    p_+(K) := \sup_{x \in K} p(x) < \infty. \label{eq:p_plus_finite}
\end{equation}
Then, for a squashing function $\Psi$, 
the set $\Sigma(\Psi)$ is dense in the variable exponent Lebesgue space $L^{p(\cdot)}(\mu)$.
\end{theorem}

In order to prove Theorem \ref{prop 20250223-1-noi}, 
we need the following two lemmas.

\begin{lemma}
\label{lemma:approx-conv}
Let $\rho \in C_{\mathrm{c}}(\mathbb{R})$ and let $\Psi$ be a squashing function. Then for every $\varepsilon > 0$ and every compact set $K \subset \mathbb{R}$, there exists a function $g \in \Sigma(\Psi)$ such that
\[
\sup_{x \in K} \left| \rho * \Psi(x) - g(x) \right| < \varepsilon.
\]
\end{lemma}

\begin{proof}
Due to the translation invariance of $\Sigma(\Psi)$,
by decomposing $\rho$ if necessary, we may assume that $\mathrm{supp}(\rho) \subset [0,1]$. Then for all $x \in K$, we can write
\[
\rho * \Psi(x) = \int_0^1 \rho(y) \Psi(x - y) \, \mathrm{d}y
= \lim_{N \to \infty} \frac{1}{N} \sum_{k = 1}^N \rho\left( \frac{k}{N} \right) \Psi\left( x - \frac{k}{N} \right).
\]
We aim to show that the convergence is uniform on $K$.

We estimate the error:
\begin{align*}
&\left| \rho * \Psi(x) - \frac{1}{N} \sum_{k = 1}^N \rho\left( \frac{k}{N} \right) \Psi\left( x - \frac{k}{N} \right) \right| \\
&\le \sum_{k = 1}^N \int_{\frac{k-1}{N}}^{\frac{k}{N}} \left| \rho(y)\Psi(x - y) - \rho\left( \frac{k}{N} \right) \Psi\left( x - \frac{k}{N} \right) \right| \, \mathrm{d}y \\
&\le \sum_{k = 1}^N \int_{\frac{k-1}{N}}^{\frac{k}{N}} \left| \rho(y) - \rho\left( \frac{k}{N} \right) \right| \left| \Psi(x - y) \right| \, \mathrm{d}y \\
&\quad + \sum_{k = 1}^N \int_{\frac{k-1}{N}}^{\frac{k}{N}} \left| \rho\left( \frac{k}{N} \right) \right| \left| \Psi(x - y) - \Psi\left( x - \frac{k}{N} \right) \right| \, \mathrm{d}y \\
&\le \sum_{k = 1}^N \int_{\frac{k-1}{N}}^{\frac{k}{N}} \left| \rho(y) - \rho\left( \frac{k}{N} \right) \right| \, \mathrm{d}y + \| \rho \|_{L^\infty} \sum_{k = 1}^N \int_{\frac{k-1}{N}}^{\frac{k}{N}} \left| \Psi(x - y) - \Psi\left( x - \frac{k}{N} \right) \right| \, \mathrm{d}y.
\end{align*}

Since $\rho$ is uniformly continuous on $[0,1]$, the first sum satisfies
\[
\sum_{k = 1}^N \int_{\frac{k-1}{N}}^{\frac{k}{N}} \left| \rho(y) - \rho\left( \frac{k}{N} \right) \right| \, \mathrm{d}y
\le \sup_{\substack{u,v \in [0,1]\\ |u - v| \le N^{-1}}} |\rho(u) - \rho(v)| = o(1),
\]
as $N \to \infty$.

Now, let $K = [A,B]$ for some integers $A,B \in \mathbb{Z}$
by expanding $K$ if necessary. For $x \in [A,B]$, we estimate the second sum:
\[
\sum_{k = 1}^N \int_{\frac{k-1}{N}}^{\frac{k}{N}} \left| \Psi(x - y) - \Psi\left( x - \frac{k}{N} \right) \right| \, \mathrm{d}y
= \sum_{k = 1}^N \int_{x - \frac{k}{N}}^{x - \frac{k-1}{N}} \left| \Psi(z) - \Psi\left( x - \frac{k}{N} \right) \right| \, \mathrm{d}z.
\]

Let $\{z_m\}_{m=1}^M$ be a partition of $[A-1, B]$ that refines $\{ A, x-1, x - \tfrac{1}{N}, \dots, x, B \}$ and has mesh at most $N^{-1}$. Then
\begin{align*}
\lefteqn{
\sum_{k = 1}^N \int_{\frac{k-1}{N}}^{\frac{k}{N}} \left| \Psi(x - y) - \Psi\left( x - \frac{k}{N} \right) \right| \, \mathrm{d}y
}\\
&\le \sum_{m=1}^M (z_m - z_{m-1}) \left( \sup_{w \in [z_{m-1}, z_m]} \Psi(w) - \inf_{w \in [z_{m-1}, z_m]} \Psi(w) \right),
\end{align*}
which tends to $0$ as $N \to \infty$ (via Darboux sums).

Thus, the convergence is uniform on $[A,B]$ due to Darboux's theorem, and there exists $N_0 \in \mathbb{N}$ such that
\[
\sup_{x \in [A,B]} \left| \rho * \Psi(x) - \frac{1}{N_0} \sum_{k = 1}^{N_0} \rho\left( \frac{k}{N_0} \right) \Psi\left( x - \frac{k}{N_0} \right) \right| < \varepsilon.
\]

Defining
\[
g(x) := \frac{1}{N_0} \sum_{k = 1}^{N_0} \rho\left( \frac{k}{N_0} \right) \Psi\left( x - \frac{k}{N_0} \right),
\]
we obtain the desired function $g \in \Sigma(\Psi)$.
\end{proof}

\begin{lemma}[{\cite[Theorem 2.4]{HSW1989}}] 
\label{lemma 20250223-1-noi}
Let $\mathcal{B}$ be the Borel $\sigma$-algebra of ${\mathbb R}^n$, 
${\Psi}$ be a squashing function, 
and let $\mu$ be a probability measure on $({\mathbb R}^n, \mathcal{B})$.
Then $\Sigma(\Psi)$ is uniformly dense on compacta in $C({\mathbb R}^n)$.
That is, for all compact subsets $K \subset {\mathbb R}^n$, all $f \in C({\mathbb R}^n)$, and all $\varepsilon > 0$, there exists 
$g \in \Sigma(\Psi)$ such that
\[
\sup_{x \in K} |f(x) - g(x)| < \varepsilon.
\]
\end{lemma}

\begin{proof}
We provide an original proof once again.
Due to the translation invariance of $\Sigma(\Psi)$, we may assume $\Psi(0) \neq 0$. 
Let $\rho \in C^\infty_{\rm c}({\mathbb R})$ be a non-zero, non-negative function supported on a small interval, satisfying $\|\rho\|_{L^1} = 1$. 
Then, we note that $\rho * \Psi$ is a squashing function in $C^{\infty}({\mathbb R})$. 
Observe that from Lemma \ref{lemma:approx-conv},  it is sufficient to show that $\Sigma(\rho*\Psi)$ is uniformly dense on compacta in $C({\mathbb R}^n)$ to prove Lemma \ref{lemma 20250223-1-noi}. To show this, let us show that $\Sigma(\rho*\Psi)$ is dense in $C(K)$ for any compact set $K\subset {\mathbb R}^n$. 

Assume that a linear functional $I: C(K) \to {\mathbb C}$ annihilates $\overline{\Sigma(\rho*\Psi)}$, that is,
\[
I(f)=0 \quad (f \in \Sigma(\rho*\Psi)).
\]
Since $I$ is represented by a measure $\mu$ with bounded variation, it follows that
\begin{equation}\label{eq:convolution_integral_zero}
\int_K (\rho * \Psi)(w \cdot x + b) \mathrm{d}\mu(x) = 0, 
\end{equation}
where $w\in{\mathbb R}^n$ and $b\in{\mathbb R}$. 

Since $K$ is compact, differentiating \eqref{eq:convolution_integral_zero} with respect to $w$ gives
\begin{equation}\label{eq:derivative_integral_zero}
\int_K x^\alpha (\rho*\Psi)^{(|\alpha|)}(w \cdot x + b) \mathrm{d}\mu(x) = 0,
\end{equation}
for all multi-indices $\alpha$.
Since $\rho*\Psi$ is a squashing function, there exists $b = b_\alpha$ such that $(\rho*\Psi)^{(|\alpha|)}(b) \neq 0$.
Setting $b = b_\alpha$ and $w = 0$ in \eqref{eq:derivative_integral_zero}, we obtain
\[
\int_K x^\alpha \mathrm{d}\mu(x) = 0.
\]
Using the Taylor expansion of $e^{i \xi \cdot x}$, it follows that
\[
\int_K e^{i \xi \cdot x} \mathrm{d}\mu(x) = 0,
\]
for all $\xi \in {\mathbb R}^n$.
Thus, the Fourier transform of $\mu$ vanishes, implying that $\mu$ itself is identically zero and that $I(f)=0$ $(f\in C(K))$. 
This togather with the Hahn--Banach theorem shows that $\Sigma(\rho*\Psi)$ is dense in $C(K)$. 
\end{proof}

\begin{proof}[Proof of Theorem \ref{prop 20250223-1-noi}]
Take $f \in L^{p(\cdot)}({\rm d}\mu)$ arbitrary 
and let $\varepsilon > 0$. 
By virtue of 
Lemma \ref{lem:approx-convolution}, 
we can take
$g\in C_{\rm c}(\mathbb{R}^n)$ so that 
\begin{equation} \label{eq:density}
\|  f-g \|_{L^{p(\cdot)}({\rm d}\mu)}
\leq \frac{\varepsilon}{2}.
\end{equation}
By Lemma \ref{lemma 20250223-1-noi},
there exists $h \in \Sigma(\Psi)$ such that 
\begin{equation*} \label{eq:uniform_bound}
    \sup_{x \in K} |g(x) - h(x)| \leq \frac{\varepsilon}{2}.
\end{equation*}
This inequality and the assumption that $\mu$ is a probability measure with $\mu(K) = 1$ imply that for any $\lambda > 0$,
\begin{align*} \label{eq:modular_bound}
&\int_{{\mathbb R}^n} \left( \frac{|g(x) - h(x)|}{\lambda} \right)^{p(x)} {\rm d}\mu(x) \notag \\ 
&= 
\int_{{\mathbb R}^n \setminus K} \left( \frac{|g(x) - h(x)|}{\lambda} \right)^{p(x)} {\rm d}\mu(x)
+
\int_K \left( \frac{|g(x) - h(x)|}{\lambda} \right)^{p(x)} {\rm d}\mu(x)  \notag \\ 
&\leq \int_K \left( \frac{\varepsilon}{2\lambda} \right)^{p(x)} {\rm d}\mu(x). 
\end{align*}
Taking $\displaystyle \lambda=\frac{\varepsilon}{2}$, we obtain 
\begin{align*}
\int_{{\mathbb R}^n} \left( \frac{|g(x) - h(x)|}{\lambda} \right)^{p(x)} {\rm d}\mu(x) 
\leq \int_K \left( \frac{\varepsilon}{2\lambda} \right)^{p(x)} {\rm d}\mu(x)=1, 
\end{align*}
which leads us to
\begin{equation} \label{eq:modular_final}
\| g-h \|_{L^{p(\cdot)}({\rm d}\mu)}
\leq \lambda=\frac{\varepsilon}{2}.
\end{equation}
Combining \eqref{eq:density} and \eqref{eq:modular_final}, 
we conclude that 
$
\| f-h \|_{L^{p(\cdot)}({\rm d}\mu)}
\le  \varepsilon$.
\end{proof}

\subsection{Park--Sandberg \cite{Park1993} (1993)}
\label{s3.2}

To describe the result of Park and Sandberg,
we give the following definition:
\begin{definition}
Let $\phi : \mathbb{R}^n \to \mathbb{R}$ be an integrable function, i.e., $\phi \in L^1(\mathbb{R}^n)$. We define the set $S_1$ as the collection of all functions $f : \mathbb{R}^n \to \mathbb{R}$ that can be expressed in the form
\begin{equation}\label{eq:S1_definition}
    f(x) = \sum_{j=1}^{N} w_j \phi \left( \frac{x - z_j}{\sigma_j} \right), \quad x \in \mathbb{R}^n,
\end{equation}
where $N \in \mathbb{N}$, $\sigma_j > 0$, $w_j \in \mathbb{R}$, and $z_j \in \mathbb{R}^n$ for each $j = 1, 2, \dots, N$.
\end{definition}

We extend the result of \cite[Proposition 1]{Park1993} to the setting with a variable exponent.

\begin{theorem}\label{Theorem-main-Park}
    Let $\phi: \mathbb{R}^n \to \mathbb{R}$ be a function such that $\phi \in L^1 \cap L^{p(\cdot)}$ and satisfies
    \begin{equation}\label{eq:phi_integral}
        \int_{\mathbb{R}^n} \phi(x)\,{\rm d}x = 1.
    \end{equation}
    Suppose that $p(\cdot): \mathbb{R}^n \to (1,\infty)$ satisfies
    \begin{equation}\label{eq:p_bounds}
        1 < p_- \leq p_+ < \infty.
    \end{equation}
    Then the set $S_1$ is dense in $L^{p(\cdot)}$.
\end{theorem}

\begin{proof}
Assume, for contradiction, that $\overline{S_1} \subsetneq L^{p(\cdot)}$.
By the Hahn--Banach theorem, there exists $T \in (L^{p(\cdot)})^*$ such that
\[ T(\overline{S_1}) = \{ 0 \}, \quad T(L^{p(\cdot)}) \neq \{ 0 \}. \]
By the Riesz representation theorem, there exists $g \in L^{p'(\cdot)}$ such that
\begin{equation}\label{eq:T-representation}
T(f) = \int_{\mathbb{R}^n} f(x)g(x){\rm d}x, \quad \text{for all } f \in L^{p(\cdot)}.
\end{equation}
Fix $z \in \mathbb{R}^n$ and $\sigma > 0$. Since $T(\overline{S_1}) = \{ 0 \}$, we have
\begin{equation}\label{eq:zero-integral}
\int_{\mathbb{R}^n} \sigma^{-n} \phi \left(  \frac{x-z}{\sigma}\right) g(x){\rm d}x = 0.
\end{equation}
Define the functions $\widetilde{\phi}$ and $\widetilde{\phi_\sigma}$ by
\[ \widetilde{\phi}(x) = 
\left( \int_{\mathbb{R}^n} \phi(y)\,{\rm d}y \right)^{-1} \phi(-x), \]
\[ \widetilde{\phi_\sigma}(x) = 
\sigma^{-n} \widetilde{\phi}\left( \frac{x}{\sigma} \right). \]
Then, we obtain
\begin{equation}\label{eq:convolution-zero}
(\widetilde{\phi_\sigma} * g)(z) = 
\left( \int_{\mathbb{R}^n} \phi(y)\, {\rm d}y \right)^{-1} \int_{\mathbb{R}^n} \sigma^{-n} \phi\left(  \frac{x-z}{\sigma}\right) g(x){\rm d}x = 0.
\end{equation}
Since $\widetilde{\phi} \in L^1$ and 
\[ \int_{\mathbb{R}^n}  \widetilde{\phi}(x){\rm d}x = 1, \]
Lemma \ref{lem:approx-convolution}
implies that
\begin{equation}\label{eq:approximation}
\lim_{\sigma \to +0} \left\| \widetilde{\phi_\sigma} * g - g  \right\|_{L^{p'(\cdot)}} = 0.
\end{equation}
Combining \eqref{eq:convolution-zero} and \eqref{eq:approximation}, we deduce that $g(x) = 0$ for almost every $x \in \mathbb{R}^n$, implying that $T(f) = 0$ for all $f \in L^{p(\cdot)}$. This contradicts $T(L^{p(\cdot)}) \neq \{ 0 \}$, completing the proof.
\end{proof}

\subsection{Approximation theorems in Herz spaces}
\label{s3.3}


The space $L^0({\mathbb R}^n)$ denotes the space of all equivalence classes of 
measurable functions modulo null functions.
Given $r>0$, we write
$B(r)=\{ x\in \mathbb{R}^n \, : \, |x|<r \}$.
For each
$k \in {\mathbb Z}$,
we set
${C}_k:=B(2^k) \setminus B(2^{k-1})$.
We write
$\chi_k=\chi_{C_k}$
for each $k \in {\mathbb Z}$.

\begin{definition}
Let
$p,q \in [1,\infty]$ and $\alpha\in \mathbb{R}$. 
The (non-homogeneous) Herz space $K_{p}^{\alpha,q}({\mathbb R}^n)$ consists of all 
$f \in L^0({\mathbb R}^n)$ for which
\[
\| f\|_{K_{p}^{\alpha,q}} := 
\| f\|_{L^{p}(B(1)) } 
+ 
\|  \{
2^{k\alpha} \|f\chi_k\|_{L^p}
\}_{k=1}^\infty
\|_{\ell^q(\mathbb{N})}
<\infty  .
\]
\end{definition}

Let $\mathcal{S}$ be the Schwartz class.
It is proved in \cite{FPL2017}
that $\mathcal{S} \subset K_{p}^{\alpha,q}({\mathbb R}^n)$.
Moreover the inclusion is dense,
provided that $p,q \in [1,\infty)$.
In order to obtain
a density theorem in Herz spaces
we need the following approximation.

\begin{lemma}[Theorem 3.6 in \cite{approxHerz2017}]
\label{lemma-Herz-approx}
Let $\phi: \mathbb{R}^n \to \mathbb{R}$ satisfy $\phi \in \mathcal{S}$
and $\displaystyle \int_{\mathbb{R}^n} \phi(x){\rm d}x=1$. 
Suppose that
$p,q \in [1,\infty)$ and $\alpha \in \mathbb{R}$.
Then we have
that for all $f\in K_{p}^{\alpha,q}({\mathbb R}^n)$,
\[
\lim_{\sigma \to +0} 
\left\| \phi_\sigma*f-f  \right\|_{K_{p}^{\alpha,q}({\mathbb R}^n)}  
=0.
\]
\end{lemma}

We note that the dual space of $K_{p}^{\alpha,q}(\mathbb{R}^n)$ is given by
\begin{equation}
    \label{eq:dual-space}
    \big(K_{p}^{\alpha,q}(\mathbb{R}^n)\big)^* = K_{p'}^{-\alpha,q'}(\mathbb{R}^n).
\end{equation}
Applying Lemma \ref{lemma-Herz-approx} in the proof of Theorem \ref{Theorem-main-Park}, we obtain a density theorem in non-homogeneous Herz spaces.

\begin{theorem}\label{theorem-density}
    Let $\phi: \mathbb{R}^n \to \mathbb{R}$ be a function satisfying $\phi \in \mathcal{S}(\mathbb{R}^n)$ and
    \begin{equation}
        \label{eq:phi-integral}
        \int_{\mathbb{R}^n} \phi(x){\rm d}x = 1.
    \end{equation}
    Suppose that $p, q \in (1,\infty)$ and $\alpha \in \mathbb{R}$. Then, the set $S_1$ is dense in $K_{p}^{\alpha,q}(\mathbb{R}^n)$.
\end{theorem}

\subsection{A modular inequality in $L^{p(\cdot)}$}
\label{s3.4}

Thanks to Lemma \ref{lem:approx-convolution}, we can approximate each function $f \in L^{p(\cdot)}$ using the convolution $\phi_\sigma * f$. In this section, we study the boundedness of the operator
\begin{equation}\label{eq:convolution_operator}
\phi_\sigma* : f  \mapsto  \phi_\sigma * f
\end{equation}
on $L^{p(\cdot)}$.

Given a measurable function $f$ on $\mathbb{R}^n$, we define the Hardy--Littlewood maximal operator $M$ by
\begin{equation}\label{eq:maximal_operator}
M f(x) := \sup\limits_{r>0} \frac{1}{r^n} \int_{|y-x|<r} |f(y)|{\rm d}y, \quad x \in \mathbb{R}^n.
\end{equation}
The classical Hardy--Littlewood maximal theorem states that $M$ is bounded on $L^p$ for $1 < p \leq \infty$. Moreover, $M$ is also bounded on $L^{p(\cdot)}$ provided that the function $p(\cdot)$ satisfies the following conditions:
\begin{align}
\left| \frac{1}{p(x)} - \frac{1}{p(y)} \right| &\le \frac{C}{-\log(|x-y|)}, \quad \text{for } |x-y| \leq \frac{1}{2}, \label{eq:p_condition1} \\
\left| \frac{1}{p(x)} - \frac{1}{p(y)} \right| &\le \frac{C}{\log(e + |x|)}, \quad \text{for } |y| \geq |x|. \label{eq:p_condition2}
\end{align}
For details, we refer to \cite{CDF2009,CFN2003,Diening2004,DHHMS}.

The boundedness of $M$ on $L^{p(\cdot)}$ implies the norm inequality
\begin{equation}\label{eq:M_bounded}
\| M f \|_{L^{p(\cdot)}} \leq C \| f \|_{L^{p(\cdot)}} \quad \text{for all } f \in L^{p(\cdot)}.
\end{equation}
Lerner \cite{Lerner-modular} showed that if $M$ satisfies the modular inequality
\begin{equation}\label{eq:M_modular}
\int_{\mathbb{R}^n} \{ M f(x) \}^{p(x)} {\rm d}x \leq 
C \int_{\mathbb{R}^n} |f(x)|^{p(x)} {\rm d}x
\quad \text{for all } f \in L^{p(\cdot)},
\end{equation}
then $p(\cdot)$ must be a constant. Note that the norm inequality \eqref{eq:M_bounded} and the modular inequality \eqref{eq:M_modular} are equivalent
if $p(\cdot)$ is a constant. This result highlights a fundamental difference between constant and variable exponent Lebesgue spaces.

Suppose $\phi \in L^1$ and $\int_{\mathbb{R}^n} \phi(x) {\rm d}x = 1$. 
To study the modular inequality for the operator $\phi_\sigma*$, we introduce the following class of functions:

\begin{definition}
A function $\Phi$ is called \emph{radial decreasing} if it satisfies $\Phi(x) \geq \Phi(y) \geq 0$ for all $x, y \in \mathbb{R}^n$ with $|x| \leq |y|$. The class $\mathcal{RB}$ consists of functions $\phi$ for which there exists a radial decreasing function $\Phi$ such that $|\phi(x)| \leq \Phi(x)$
for all $x \in {\mathbb R}^n$, where $\Phi(0) < \infty$ and $\Phi \in L^1$.
\end{definition}

Assuming $\phi \in \mathcal{RB}$, it follows from \cite[Proposition 2.7]{Duoa} that
\begin{equation}\label{eq:phi_maximal_bound}
\sup_{\sigma > 0} |\phi_\sigma * f(x)| \leq C M f(x) \quad \text{for all } x \in \mathbb{R}^n.
\end{equation}
Thus, if $M$ is bounded on $L^{p(\cdot)}$, then $\phi_\sigma*$ is also bounded on $L^{p(\cdot)}$ with an operator norm independent of $\sigma$. This leads to the following generalization of Lerner's result.

\begin{theorem}\label{thm:phi_modular}
Let $p(\cdot) : \mathbb{R}^n \to (1,\infty)$ be a measurable function. Suppose $\phi \in L^1 \cap \mathcal{RB}$, satisfying
\begin{itemize}
    \item $\int_{\mathbb{R}^n} \phi(x) \, {\rm d}x = 1$,
    \item $\phi(0) > 0$,
    \item $\phi$ is continuous on $\mathbb{R}^n$.
\end{itemize}
If the modular inequality
\begin{equation}\label{eq:phi_modular}
    \int_{\mathbb{R}^n} |\phi_\sigma * f(x)|^{p(x)} \, {\rm d}x 
    \leq C \int_{\mathbb{R}^n} |f(x)|^{p(x)} \, {\rm d}x,
\end{equation}
holds for all $f \in L^{p(\cdot)}$ and all $\sigma > 0$, then $p(\cdot)$ must be constant.
\end{theorem}

\begin{proof}
The proof follows arguments found in \cite{ABM,INS2015}. 
Assuming the modular inequality
$(\ref{eq:phi_modular})$
holds for all $f \in L^{p(\cdot)}$ and all $\sigma > 0$, we derive a contradiction under the assumption that $p(\cdot)$ is not constant.

By the continuity of $\phi$ and the assumption $\phi(0) > 0$, there exists a constant $C_0 > 0$ and $j \gg 1$ such that
\begin{equation}
\phi(z) > C_0, \quad 
\text{for } |z| < 2^{-j}.
\end{equation}

If $p(\cdot)$ is not constant, then $p_+ > p_-$ holds. Define
\[
    \varepsilon := \frac{1}{3} (p_+ - p_-), \quad E := \{x \in \mathbb{R}^n : p_+ - \varepsilon < p(x)\}, \quad F := \{x \in \mathbb{R}^n : p_- + \varepsilon > p(x)\}.
\]
Note that $\varepsilon > 0$, $|E| > 0$ and $|F| > 0$.
By the Lebesgue differentiation theorem,
\[
\lim_{r \to +0}
\frac{|B(y_0,r) \cap E|}{|B(y_0,r)|}=1
\]
for almost all $y_0 \in E$ and
\[
\lim_{r \to +0}
\frac{|B(x_0,r) \cap F|}{|B(x_0,r)|}=1
\]
for almost all
$x_0 \in F$. 
Choose $r\in(0,2^{-j})$, $y_0 \in E$ and $x_0 \in F$ so that
\[
\frac{|B(y_0,r) \cap E|}{|B(y_0,r)|}>\frac12, \quad 
\frac{|B(x_0,r) \cap F|}{|B(x_0,r)|}>\frac12.
\]
Then $|E \cap B(y_0,r)| > 0$ and $|F \cap B(x_0,r)| > 0$
in particular. We additionally define
\[
U := E \cap B(y_0,r), \quad V := F \cap B(x_0,r).
\]

Assume that $\sigma>2^{j}(|x_0-y_0|+2r)$.
Then
\[
    \phi_\sigma*\chi_V(x) = 
\frac{1}{\sigma^n}\int_V \phi\left(\frac{x - y}{\sigma}\right) \, {\rm d}y \geq\frac{C_0|V|}{\sigma^n}, \quad x \in U.
\]
since
\[
    |x - y| \leq |x - y_0| + |y_0 - x_0| + |x_0 - y| \leq 2r + |y_0 - x_0|,
\]
for every $x \in U$ and $y \in V$. 

Taking an arbitrary constant $R > 1$, we obtain
\begin{align*}
    R^{p_+ - \varepsilon} \int_U 
    \left( \frac{C_0|V|}{\sigma^n} \right)^{p(x)} \, {\rm d}x 
    &\leq \int_U \left( C_0 R |V| \right)^{p(x)} \, {\rm d}x \\
    &\leq \int_U 
\left( \frac{R C_0|V|}{\sigma^n} \right)^{p(x)} \, {\rm d}x \\
    &\leq \int_{\mathbb{R}^n} \left( R P_j (\chi_V)(x) \right)^{p(x)} 
    \, {\rm d}x.
\end{align*}

Applying inequality \eqref{eq:phi_modular}, we get
\begin{align*}
R^{p_+ - \varepsilon} \int_U \left( C_0 2^{-n(j+1)} \right)^{p(x)} \, {\rm d}x \leq C \int_{\mathbb{R}^n} \left( R \chi_V(x) \right)^{p(x)} \, {\rm d}x 
= C \int_V R^{p(x)} \, {\rm d}x 
\leq C R^{p_- + \varepsilon} |V|.
\end{align*}
This contradicts the fact that $p_+ - \varepsilon > p_- + \varepsilon$. Hence, we conclude that $p_+ = p_-$, which means that the variable exponent $p(x)$ is constant.
\end{proof}

\section*{Acknowledgment}
This work was partly supported by MEXT Promotion of Distinctive Joint Research Center Program JPMXP0723833165 
and 
Tokyo City University Start-up Support for Interdisciplinary Research.
 This work was supported by Grant-in-Aid for Scientific Research (C) (19K03546,
 23K03156), the Japan
 Society for the Promotion of Science (Sawano).

\end{document}